\documentclass{amsart}%
\usepackage{amssymb}
\usepackage{amsmath}
\usepackage{amsfonts}
\usepackage{geometry}
\usepackage{graphicx}%
\providecommand{\U}[1]{\protect\rule{.1in}{.1in}}
\newtheorem{theorem}{Theorem}[section]
\theoremstyle{plain}

\newtheorem{corollary}[theorem]{Corollary}

\newtheorem{definition}[theorem]{Definition}

\newtheorem{lemma}[theorem]{Lemma}

\newtheorem{remark}[theorem]{Remark}

\numberwithin{equation}{section}
\begin{document}
\title{On the error estimate for cubature on Wiener space}
\author{Thomas Cass}
\address{Mathematical Institute and Oxford-Man Institute of Quantitative Finance,
University of Oxford }
\author{Christian Litterer}
\address{Institut f\"ur Mathematik, Technische Universit\"at Berlin and Oxford-Man
Institute of Quantitative Finance, University of Oxford }
\thanks{The research of Thomas Cass is supported by EPSRC grant EP/F029578/1 and a
Career Development Fellowship at Christ Church, Oxford}
\thanks{The research of Christian Litterer is supported by the European Research
Council under the European Union's Seventh Framework Programme (FP7/2007-2013)
/ ERC grant agreement nr. 258237.}

\begin{abstract}
It was pointed out in Crisan, Ghazali \cite{crisan} that the error estimate
for the cubature on Wiener space algorithm developed in Lyons, Victoir
\cite{lyons} requires an additional assumption on the drift. In this note we
demonstrate that it is straightforward to adopt the analysis of Kusuoka
\cite{kusuoka} to obtain a general estimate without an additional assumptions
on the drift. In the process we slightly sharpen the bounds derived in
\cite{kusuoka}.

\end{abstract}
\maketitle

\section{Introduction}

In pricing and hedging financial derivatives as well as in assessing the risk
inherent in complex systems we often have to find approximations to
expectations of functionals of solutions to stochastic differential equations
(SDE). We consider a Stratonovich stochastic differential equation%
\begin{equation}
d\xi_{t,x}=V_{0}(\xi_{t,x})dt+\sum_{i=1}^{d}V_{i}(\xi_{t,x})\circ dB_{t}%
^{i},\;\xi_{0,x}=x \label{sde}%
\end{equation}
defined by a family of smooth vector fields $V_{i}$ and driven by Brownian
motion. It is well known that computing $P_{T-t}f:=E(f(\xi_{T-t,x}))$
corresponds to solving a parabolic partial differential equation (PDE). The
cubature on Wiener space method developed by Lyons, Victoir in \cite{lyons},
following Kusuoka \cite{kusuoka3} (in the following also referred to as the
KLV method) is a high order particle method for approximating the weak
solution of stochastic differential equations in Stratonovich form. To obtain
high order error bounds the test functions are assumed to be Lipschitz and the
vector fields defining the SDE satisfy Kusuoka's UFG condition (see
\cite{kusuoka2}), which is a weaker assumption than the usual uniform
H\"{o}rmander condition.

High order particle methods have since been shown to be highly effective in
practice see e.g. \cite{ninomiya}, \cite{ninomiya-victoir} and further
extensions and applications of cubature on Wiener space have been developed by
various authors. Applications include the non-linear filtering problem
\cite{crisan}, stochastic backward differential equations \cite{CM},
\cite{CM1}, calculating Greeks by cubature methods \cite{teichmann} or
extending the KLV method by adding recombination in \cite{litterer}. It was
pointed out in Crisan, Ghazali \cite{crisan} that the analysis of the error
bounds in Lyons, Victoir \cite{lyons} requires an additional assumption on the
drift (see Definition \ref{V0-condition}) and the question was raised if this
additional assumption is necessary to derive high order error bounds. We first
give a brief introduction to cubature on Wiener space and outline how the need
for an additional assumption on the drift arises in \cite{lyons}. Then, based
on Litterer \cite{litterer}, we demonstrate carefully how the analysis in
Kusuoka \cite{kusuoka} can be adopted to derive similar bounds for cubature on
Wiener space. We show for the KLV method based on a cubature measure of degree
$m$ over a $k$ step partition $\mathcal{D}$ the error $E_{\mathcal{D}}$ can be
bounded by
\[
E_{\mathcal{D}}:=\sup_{x\in R^{N}}\left\vert E_{P}f(\xi_{x,s})-E_{Q}%
f(\xi_{x,s})\right\vert \leq C\left(  \sum_{j=m+1}^{2m}s^{j/2}\Vert
f\Vert_{V,j}+s^{(m+1)/2}\Vert\nabla f\Vert_{\infty}\right)
\]
for any $s\in(0,1]$. Note that these bounds do not contain any higher order
derivatives in the direction of the drift $V_{0}$ and, although our proof
contains many elements of the analysis of a version of Kusuoka's algorithm
carried out in \cite{kusuoka}, we obtain slightly sharper error bounds in the
process involving $2m$ instead of $m^{m+1}$ derivatives. For suitable families
of partitions (first considered in Kusuoka \cite{kusuoka3}) the error bounds
immediately lead to convergence of order $\left(  m-1\right)  /2$ in the
number of time steps in the partition. Finally, we clarify the relation of the
KLV\ method to the version of Kusuoka's algorithm analysed in \cite{kusuoka}.

\section{Cubature measures}

Let $C_{b}^{\infty}(R^{N},R^{N})$ denote the smooth bounded $R^{N}$ valued
functions whose derivatives of all order are bounded. Then $V_{i}=(V_{i}%
^{1},\ldots,V_{i}^{N})\in C_{b}^{\infty}(R^{N},R^{N}),\,0\leq i\leq d$ may be
identified with smooth vector fields on $R^{N}$. Let $B=(B_{t}^{1}%
,\ldots,B_{t}^{d})$ be a Brownian motion and $B_{t}^{0}(t)=t$. Let $\xi_{t,x}%
$, $t\in\lbrack0,T]$, $x\in R^{N}$ be a version of the solution of the
Stratonovich stochastic differential equation (SDE) (\ref{sde}) that coincides
with the pathwise solution on continuous paths of bounded variation (recall
that the set of bounded variation paths have zero Wiener measure). We define
the It\^{o} functional $\Phi_{T,x}:C_{0}^{0}([0,T],R^{d})\rightarrow R^{N}$
by
\[
\Phi_{T,x}(\omega)=\xi_{T,x}(\omega).
\]
The particular choice for the version of the SDE solution when defining
$\xi_{t,x}$ implies that the It\^{o} functional for a bounded variation path
$\omega$ coincides with the usual ODE solution of equation (\ref{sde}) along
the path $\omega$.

Define the set of all multi-indices $A$ by $A=\bigcup_{k=0}^{\infty}%
\{0,\ldots,d\}^{k}$ and let $\alpha=(\alpha_{1},\ldots,\alpha_{k})\in A$ be a
multi-index. Furthermore we define a degree on a multi-index $\alpha$ by
$\Vert\alpha\Vert=k+card\left\{  j:\alpha_{j}=0\right\}  $ and
\[
A(j)=\{\alpha\in A:\Vert\alpha\Vert\leq j\}.
\]
Let $A_{1}=A\setminus\{\emptyset,(0)\}$ and $A_{1}(j)=\{\alpha\in A_{1}%
:\Vert\alpha\Vert\leq j\}$. Following Kusuoka \cite{kusuoka} we inductively
define a family of vector fields indexed by $A$ by taking
\[
V_{[\emptyset]}=0,\quad V_{[i]}=V_{i},\quad0\leq i\leq d
\]%
\[
V_{[(\alpha_{1},\ldots\alpha_{k},i)]}=[V_{[\alpha]},V_{i}],\quad0\leq i\leq
d,\alpha\in A.
\]
Moreover let $V_{\alpha}=V_{\alpha_{1}}\cdots V_{\alpha_{k}}$ where the
composition is taken in the sense of differential operators. Finally we define
a family of semi-norms on the space of functions $C_{b}^{\infty}(R^{N})$
\[
\Vert f\Vert_{V,k}=\sum_{j=1}^{k}\sum_{\alpha_{1},\ldots,\alpha_{j}\in
A_{1},\Vert\alpha_{1}\Vert+\dots+\Vert\alpha_{j}\Vert=k}\Vert V_{[\alpha_{1}%
]}\dots V_{[\alpha_{j}]}f\Vert_{\infty}.
\]
It is important to note that these semi-norms contain no derivatives in the
direction of $V_{0}$. For $V\in C_{b}^{\infty}(R^{N};R^{N})$ we define the
flow $Exp(tV)(x)$ to be the solution of the autonomous ODE
\[
\dot{X}(t,x)=V(X(t,x))\quad t>0,\qquad X(0,x)=x\in R^{N}.
\]

A cubature measure on a finite dimensional measure space is a discrete
positive measure that integrates polynomials up to a certain (finite) degree
correctly (i.e. as under Wiener measure). Together with the Taylor
approximation for error estimation, cubature is a classical and efficient
approach to the numerical integration of sufficiently smooth functions. For
the Wiener space setting \cite{lyons} a cubature measure is a discrete measure
supported on paths of bounded variation and the role of polynomials is taken
by the analogous Wiener functionals (iterated Stratonovich integrals).

\begin{definition}
For fixed $T>0$ we say that a discrete measure $Q_{T}$ assigning positive
weights $\lambda_{1},\ldots,\lambda_{n}$ to paths
\[
\omega_{1},\ldots,\omega_{n}\in C_{0,bv}^{0}([0,T],R^{d})
\]
is a cubature measure of degree $m$, if for all $(i_{1},\ldots,i_{k})\in
A\left(  m\right)  $,
\begin{equation}
E\left(  \int_{0<t_{1}<\dots<t_{k}<T}\circ dB_{t_{1}}^{i_{1}}\dots\circ
dB_{t_{k}}^{i_{k}}\right)  =\sum_{j=1}^{n}\lambda_{j}\int_{0<t_{1}<\dots
<t_{k}<T}d\omega_{j}^{i_{1}}(t_{1})\dots d\omega_{j}^{i_{k}}(t_{k}),
\label{cub-def-id}%
\end{equation}
where the expectation is taken under Wiener measure.
\end{definition}

By the scaling property of Brownian motion any cubature measure $Q_{T}$ may be
obtained from $Q_{1}$ by letting $\omega_{T,i}^{j}(t)=\sqrt{T}\omega_{i}%
^{j}(t/T)$, $j=1,\ldots,d$ and keeping the weights of $Q_{1}$.

Taylor expansions play a crucial role in the estimation of the error when we
replace the original (Wiener) measure by a cubature measure. On Wiener space
the bounds for sufficiently smooth functions are obtained by considering
stochastic Taylor expansion. The following proposition is a sharpened version
of Proposition 2.1 in \cite{lyons}.

\begin{lemma}
\label{p1.1.1} Let $f\in C_{b}^{\infty}\left(  R^{N}\right)  $, $m\in N$. Then
for every $t>0$
\begin{equation}
f(\xi_{t,x})=\sum_{(\alpha_{1},\ldots,\alpha_{k})\in A(m)}V_{\alpha_{1}}\cdots
V_{\alpha_{k}}f(x)\int_{0<t_{1}<\cdots<t_{k}<t}\circ dB_{t_{1}}^{\alpha_{1}%
}\dots\circ dB_{t_{k}}^{\alpha_{k}}+R_{m}(t,x,f). \label{STremainder}%
\end{equation}
And the remainder process $R_{m}(t,x,f)$ satisfies
\[
\sup_{x\in R^{N}}\sqrt{E(R_{m}(t,x,f)^{2})}\leq C\sum_{j=m+1}^{m+2}t^{j/2}%
\sup_{(\alpha_{1},\ldots,\alpha_{i})\in A(j)\setminus A(j-1)}\Vert
V_{\alpha_{1}}\dots V_{\alpha_{i}}f\Vert_{\infty},
\]
where $C$ is a constant only depending on $d$ and $m$.
\end{lemma}

\begin{proof}
By induction one can prove that the remainder $R_{m}(t,x,f)$ of the
Stratonovich stochastic Taylor expansion is given by
\[
R_{m}(t,x,f)=\sum_{\overset{(\alpha_{2},\ldots,\alpha_{k})\in A(m)}{(\alpha
_{1},\ldots,\alpha_{k})\notin A(m)}}\int_{0<t_{1}<\dots<t_{k}<t}V_{\alpha_{1}%
}\dots V_{\alpha_{k}}f(\xi_{t_{1},x})\circ dB_{t_{1}}^{\alpha_{1}}\dots\circ
dB_{t_{k}}^{\alpha_{k}}.
\]
The proposition follows from an elementary calculation using the It\^{o}
formula (see Litterer \cite{thesis} for details).
\end{proof}

The following lemma is the analogous of Proposition \ref{p1.1.1} for the
cubature measures $Q_{T}$ and its proof may be found in Lyons, Victoir
\cite{lyons}.

\begin{lemma}
\label{lem1.1} \label{l1.1.2} Let $R_{m}(T,x,f)$ be the process defined in
(\ref{STremainder}) then we have
\[
\sup_{x \in R^{N}} E_{Q_{T}} | R_{m}(T,x,f) | \leq C(d,m,Q_{1}) \sum
_{j=m+1}^{m+2} T^{j/2} \sup_{(\alpha_{1},\ldots, \alpha_{i}) \in A(j)
\setminus A(j-1)} \| V_{\alpha_{1}} \dots V_{\alpha_{i}} f\|_{\infty},
\]
where C is a constant depending only on $d$, $m$ and the length of the bounded
variation paths in the support of the cubature measure $Q_{1}$.
\end{lemma}

The constant in Lemma \ref{l1.1.2} can in fact be made explicit (see Crisan,
Ghazali \cite{crisan} Example 4).\newline The expectation of the Taylor
approximation $f(\xi_{t,x})-R_{m}(s,x,f)$ defined in $\left(
\ref{STremainder}\right)  $ under Wiener and cubature measure coincide by
definition of the cubature measure. Hence, one may apply the triangle
inequality to Lemmas \ref{p1.1.1} and \ref{lem1.1} and deduce%
\begin{equation}
\sup_{x\in R^{N}}\left\vert E\left(  f(\xi_{t,x})\right)  -E_{Q_{T}}\left(
f(\xi_{t,x})\right)  \right\vert \leq C\sum_{j=m+1}^{m+2}s^{j/2}\sup
_{(\alpha_{1},\ldots,\alpha_{i})\in A(j)\setminus A(j-1)}\Vert V_{\alpha_{1}%
}\dots V_{\alpha_{i}}f\Vert_{\infty}. \label{rec-eq1}%
\end{equation}

In general, the right hand side of the inequality in $\left(  \ref{rec-eq1}%
\right)  $ is not sufficient to directly obtain a good error bound for the
approximation of the expectation; in particular, if $f$ is only assumed to be
Lipschitz the estimate appears useless. Therefore, instead of approximating
\[
P_{T}f(x):=E(f(\xi_{T,x}))
\]
in one step, one considers a partition $\mathcal{D}$ of the interval $[0,T]$
\[
t_{0}=0<t_{1}<\ldots<t_{k}=T,
\]
with $s_{j}=t_{j}-t_{j-1}$ and solves the problem over each of the smaller
subintervals by applying the cubature method recursively. If $\tau$ and
$\tau^{\prime}$ are two path segments we denote their concatenation by
$\tau\otimes\tau^{\prime}$. For the approximation we consider all possible
concatenations of cubature paths over the subintervals, i.e. all paths of the
form $\omega_{s_{1},i_{1}}\otimes\ldots\otimes\omega_{s_{k},i_{k}}$. We define
a corresponding probability measure $\nu$ by
\[
\nu=\sum_{i_{1},\ldots,i_{k}=1}^{n}\lambda_{i_{1}}\dots\lambda_{i_{k}}%
\delta_{\omega_{s_{1},i_{1}}\otimes\ldots\otimes\omega_{s_{k},i_{k}}}.
\]
The iterated cubature method may be interpreted as a Markov operator and,
hence the error of the approximation of $P_{T}f$ by $E_{\nu}(f(\xi_{T,x}))$ is
bounded above by the sum of the errors of the approximations over the
subintervals. The error over each subintervals can in turn be bounded by
applying $\left(  \ref{rec-eq1}\right)  $ to $P_{T-t_{i}}f$ instead of $f$ and
exploiting the regularity of $P_{T-t_{i}}f.$ The following result is a
corollary to Kusuoka-Stroock \cite{kusuoka1} and Kusuoka \cite{kusuoka2} , for
a detailed proof see \cite{crisan}.

\begin{corollary}
\label{crisan-regularity} Suppose the family of vector fields $V_{i}$, $0\leq
i\leq d$ satisfy the UFG condition. Let $f\in C_{b}^{\infty}(R^{N}),$
$s\in(0,1]$ and $\alpha_{1},\ldots,\alpha_{j}\in A_{1}$ then
\begin{equation}
\Vert V_{[\alpha_{1}]}\dots V_{[\alpha_{j}]}P_{s}f\Vert_{\infty}\leq
\frac{Cs^{1/2}}{s^{(\Vert\alpha_{1}\Vert+\ldots+\Vert\alpha_{j}\Vert)/2}}%
\Vert\nabla f\Vert_{\infty}%
\end{equation}
for all $s\leq1$, where $C$ is a constant independent of $s$ and $f$.
\end{corollary}

As the regularity estimates in the previous corollary do not hold in the
$V_{0}$ direction, but the Taylor based estimates used to obtain $\left(
\ref{rec-eq1}\right)  $ require higher derivative in the $V_{0}$ direction it
was pointed out in Crisan, Ghazali \cite{crisan} that the analysis in Lyons,
Victoir \cite{lyons} requires an additional assumption on the drift, we state
this as follows.

\begin{definition}
[V0 condition]\label{V0-condition}A family of vector fields $V_{i}$, $0\leq
i\leq d$ satisfies the V0 condition if
\[
V_{0}=\sum_{\beta\in A_{1}(2)}u_{\beta}V_{[\beta]}%
\]
for some $u_{\beta}\in C_{b}^{\infty}(R^{N})$.
\end{definition}

The following theorem taken from Lyons, Victoir \cite{lyons} is the main error
estimate for the iterated cubature method.

\begin{theorem}
\label{klv-error-thm}Suppose the vector fields satisfy the UFG\ and V0
conditions then
\[
\sup_{x\in R^{N}}\left\vert P_{T}f\left(  x\right)  -E_{\nu}(f(\xi
_{T,x}))\right\vert \leq C\left(  T\right)  \Vert\nabla f\Vert_{\infty}\left(
s_{k}^{1/2}+\sum_{j=m}^{m+1}\sum_{i=1}^{k-1}\frac{s_{i}^{(j+1)/2}}%
{(T-t_{i})^{j/2}}\right)  .
\]

\end{theorem}

As an immediate Corollary one obtains (\cite{lyons} Example 14) high order
convergence of the KLV method for suitable partitions of $\left[  0,T\right]
.$

\begin{corollary}
\label{cor-rate}Consider the family of partitions given by $t_{j}=T\left(
1-\left(  1-\frac{j}{k}\right)  ^{\gamma}\right)  $ and let $v_{k}$ denote the
corresponding iterated cubature measures. Suppose the vector fields satisfy
the UFG\ and V0 conditions then%
\[
\sup_{x\in R^{N}}\left\vert P_{T}f\left(  x\right)  -E_{\nu_{k}}(f(\xi
_{T,x}))\right\vert \leq Ck^{-(m-1)/2}\Vert\nabla f\Vert_{\infty},
\]
where $C$ is a constant independent of $k$ and $f.$
\end{corollary}

In the remainder of the paper we will derive similar bounds for the KLV method
that do not require the additional V0 assumption on the drift.

\section{Algebraic preliminaries - The free Lie algebra and the signature}

\label{ch1-algebra-sect} In the following we adopt the notation of Lyons,
Victoir \cite{lyons}. Given a Banach space $W$ we define the tensor algebra of
non-commutative polynomials over $W$ by
\[
T(W):=\bigoplus_{i=0}^{\infty}W^{\otimes i}.
\]
Define $T^{(j)}(W)$ to be the quotient of $T(W)$ by the ideal $\bigoplus
_{i=j+1}^{\infty}W^{\otimes i}$. We identify $T^{(j)}(W)$ with the subspace
\[
T^{(j)}(W)=\bigoplus_{i=0}^{j}W^{\otimes i}.
\]
In the following we will not distinguish between the algebras of
non-commutative polynomials and series as we always work with their
truncations. Let $\epsilon_{0},\ldots,\epsilon_{d}$ be a fixed orthonormal
basis for $R\oplus R^{d}$. Let $T(R,R^{d})$ denote the tensor algebra of
polynomials over $R\oplus R^{d}$ endowed with a grading that assigns degree
two to $\epsilon_{0}$ and degree one to the remaining generators (see
\cite{lyons} for the details of the definition).\newline

Let $\lambda\in R$, $a=(a_{0},a_{1},\ldots)$, $b=(b_{0},b_{1},\ldots)\in
T(R,R^{d})$. Define a homogeneous scaling operation by
\[
\langle\lambda,a\rangle:=(a_{0},\lambda a_{1},\ldots,\lambda^{i}a_{i}%
,\ldots).
\]
and the exponential and logarithm on $T(R,R^{d})$ using the usual power
series. Let $\pi_{j}$ denote the natural projection of $T(R,R^{d})$ onto the
subspace $T^{(j)}(R,R^{d})$. \newline

We define a Lie bracket on $T(R,R^{d})$ by $[a,b]=a\otimes b-b\otimes a$. Let
$\mathcal{L}$ denote the free Lie algebra generated by $R\oplus R^{d}$ (see
Reutenauer \cite{reutenauer}). Then $\mathcal{L}$ is the space of linear
combinations of finite sequences of Lie brackets of elements in $W=R\oplus
R^{d}$, i.e.
\[
W\oplus\lbrack W,W]\oplus\lbrack W,[W,W]]\oplus\cdots.
\]
We call an element $u$ of $\pi_{j}(\mathcal{L})$ a Lie polynomial of degree
$j$ and an infinite sequence of Lie brackets a Lie series. Note that $\pi
_{j}(\mathcal{L})\subseteq T^{(j)}(R,R^{d})$.\newline

Words of the form $\epsilon_{\alpha}:=\epsilon_{\alpha_{1}}\otimes
\cdots\otimes\epsilon_{\alpha_{k}}$, $\alpha\in A\cup\{\emptyset\}$ form a
basis for $T(R,R^{d})$ (note that $\epsilon_{\emptyset}:=1)$. For $w_{i}%
=\sum_{\alpha\in A}w_{i\alpha}\epsilon_{\alpha}\in T(R,R^{d})$, $i=1,2$ we
define following Kusuoka \cite{kusuoka} an inner product and a norm
$\left\Vert \cdot\right\Vert _{2}$ on $T(R,R^{d})$ by
\begin{equation}
(w_{1},w_{2})=\sum_{\alpha\in A\cup\{\emptyset\}}w_{1\alpha}w_{2\alpha}%
\quad\Vert w_{1}\Vert_{2}=(w_{1},w_{1})^{1/2}. \label{inner-product}%
\end{equation}
Note that restricted to $T^{(j)}(R,R^{d})$ all norms are equivalent as
$T^{(j)}(R,R^{d})$ is finite dimensional when regarded as a vector
space.\newline

The map sending $\epsilon_{i}$ to $V_{i}$, $i=0,\ldots,d$ extends to a unique
linear map on $W$ and by the universality property of the tensor algebra
extends to a unique homomorphism $\Gamma$ from $T(R,R^{d})$ into the
differential operators on $R^{N}$. The restriction of $\Gamma$ to
$\mathcal{L}$ is a Lie map from $\mathcal{L}$ into the smooth vector fields on
$R^{N}$.

Finally we collect a number of simple algebraic facts.

\begin{lemma}
Let $w \in\mathcal{L}$ then \newline(i) The homogeneous scaling $\langle t,
\rangle$ commutes with $exp$ and $log $ \newline(ii) $\pi_{m} log (\pi_{m} w)
= \pi_{m} log (w)$ and $\pi_{m} exp (\pi_{m} w) = \pi_{m} exp (w)$
\newline(iii) $\Gamma$ restricted to $\pi_{m} \mathcal{L}$ is a linear map of
finite dimensional vector spaces and hence commutes with expectations on
$\pi_{m} \mathcal{L}$.
\end{lemma}

\begin{proof}
(i) is obvious from the definition of $log$ and $exp$ as power series. (ii)
follows from the fact that for $a,b \in T(R,R^{d})$ $\pi_{m}(\pi_{m}(a)\pi
_{m}(b))= \pi_{m}(ab)$.
\end{proof}

For a path $\phi\in C_{0,bv}^{0}([0,T],R^{d})$; $s,t\in\lbrack0,T]$ we define
its signature (also known as Chen series) $S_{s,t}:C_{0,bv}^{0}([0,T],R^{d}%
)\rightarrow T(R^{d})$ by
\begin{equation}
S_{s,t}(\phi)=\sum_{k=0}^{\infty}\int_{s<t_{1}<\cdots<t_{k}<t}d\phi
(t_{1})\otimes\cdots\otimes d\phi(t_{k}),
\end{equation}
where the summation is to be interpreted as a direct sum. Using Stratonovich
iterated integrals we may define $S_{s,t}(\circ B)$ the random Stratonovich
signature of a Brownian motion (under Wiener measure).\newline

With these definitions in mind we can restate condition $\left(
\ref{cub-def-id}\right)  $ in the definition of a cubature measure as
\begin{equation}
E(\pi_{m}(S_{0,1}(\circ B)))=\sum_{j=1}^{n}\lambda_{j}\pi_{m}(S_{0,1}%
(\omega_{j})). \label{cubature}%
\end{equation}

Chen's theorem (see e.g. Lyons, Victoir \cite{lyons} ) tells us that
$L_{i}:=\pi_{m}(log(S_{0,1}(\omega_{i})))$ is a Lie polynomial. The measure
$Q_{\mathcal{L}}=\sum_{j=1}^{n}\lambda_{j}\delta_{L_{j}}$ satisfies
\begin{equation}
E(\pi_{m}(S_{0,1}(\circ B)))=E_{Q_{\mathcal{L}}(dL)}(\pi_{m}exp(L)).
\label{cublie}%
\end{equation}
Conversely for any Lie polynomials $L_{i}$ there exist continuous bounded
variation paths $\omega_{i}$ with log-signature $L_{i}$. Moreover if
$Q_{\mathcal{L}}$ satisfies (\ref{cublie}) $Q$ will satisfy (\ref{cubature}),
so the identities (\ref{cubature}) and (\ref{cublie}) are equivalent. The
proof of Chen's theorem can be extended to show that $\log(S_{s,t}(\circ B))$
is a (random) Lie series, see e.g. Lyons \cite{lyons2}. Such arguments can be
used to obtain small time asymptotics of the solution of Stratonovich SDEs,
see e.g. Ben Arous \cite{arous}. \newline Motivated by this discussion, and
following Lyons and Victoir \cite{lyons}, we make the following equivalent
definition for a cubature measure on Wiener space.

\begin{definition}
\label{cubfreeliedef} Let $m \in\mathbb{N}$ and $Q_{\mathcal{L}}=\sum
_{j=1}^{n} \lambda_{j} \delta_{L_{j}}$ with $\lambda_{i} >0$ and $L_{i} \in
\pi_{m} (\mathcal{L})$ for $i=1, \ldots, n$. We say $Q_{\mathcal{L}}$ is a
cubature measure on Wiener space if and only if
\[
E(S^{(m)}_{0,1}(\circ B)) = E_{Q_{\mathcal{L}}(dL)} \pi_{m} \exp(L).
\]

\end{definition}

In the following we will sometimes where no confusion arises drop the
reference to the integration variable $L$ and write $E_{Q_{\mathcal{L}}}$ in
place of $E_{Q_{\mathcal{L}}(dL)}$. A cubature measure over a general time
interval $\left[  0,T\right]  $ may be obtained from $Q_{_{\mathcal{L}}}$ by
homogeneously rescaling the Lie polynomial in its support and leaving the
weights unchanged. We have
\[
E(S_{0,T}^{(m)}(\circ B))=E_{Q_{\mathcal{L}}(dL)}\pi_{m}\exp(\langle\sqrt
{T},L\rangle).
\]

\section{Error estimate for the cubature approximation}

\label{sect-general-klv} In this section we derive our main error estimate and
demonstrate that $P_{T}f$ can be approximated to high order by a cubature
measure and the bounds on the error do not involve any derivative in the
$V_{0}$ direction (but only its Lie brackets).

\begin{theorem}
\label{mainthm} Let $P$ denote the Wiener measure and $Q$ a degree $m$
cubature measure supported on paths of bounded variation. Then
\[
\sup_{x\in R^{N}}\left\vert E_{P}f(\xi_{x,s})-E_{Q}f(\xi_{x,s})\right\vert
\leq C\left(  \sum_{j=m+1}^{2m}s^{j/2}\Vert f\Vert_{V,j}+s^{(m+1)/2}%
\Vert\nabla f\Vert_{\infty}\right)
\]
for any $s\in(0,1]$, $f\in C_{b}^{\infty}(R^{N})$. The constant $C$ depends on
$d$,$m$, $Q_{1}$,\newline$\sup_{\alpha\in A(m+2)\setminus A(m)}\Vert
V_{\alpha}Id(\cdot)\Vert_{\infty}$ and $E_{P}\Vert\pi_{k}(logS(\circ
B))\Vert_{2}^{k}$, $E_{Q_{1}}\Vert\pi_{k}(logS(\circ B))\Vert_{2}^{k}%
$,\newline$1\leq k\leq2m$.
\end{theorem}

As an immediate consequence we obtain substituting Proposition 3.2 of
\cite{lyons} by Theorem \ref{mainthm} the following error estimate for the KLV
method that preserves its higher order convergence.

\begin{corollary}
With the notation of Corollary \ref{cor-rate} suppose the vector fields
satisfy the UFG\ condition then
\[
\sup_{x\in R^{N}}\left\vert P_{T}f\left(  x\right)  -E_{\nu}(f(\xi
_{T,x}))\right\vert \leq C\left(  T\right)  \Vert\nabla f\Vert_{\infty}\left(
s_{k}^{1/2}+\sum_{i=1}^{k-1}\left(  \sum_{j=m+1}^{2m}\frac{s_{i}^{j/2}%
}{(T-t_{i})^{\left(  j-1\right)  /2}}+s_{i}^{(m+1)/2}\right)  \right)
\]
and%
\[
\sup_{x\in R^{N}}\left\vert P_{T}f\left(  x\right)  -E_{\nu_{k}}(f(\xi
_{T,x}))\right\vert \leq Ck^{-(m-1)/2}\Vert\nabla f\Vert_{\infty}%
\]

\end{corollary}

To prove the theorem we will adopt the analysis of Kusuoka \cite{kusuoka} to
the cubature on Wiener space setting. In the process we sharpen the estimates
slightly allowing us to obtain a bound with at most $2m$ derivatives instead
of $m^{m+1}$ (compare Kusuoka \cite{kusuoka} Lemma 18). Recall that $Id$ is
the identity function on $R^{N}$ defined by $Id(x)=x$.

Before going into technical details we give an interpretation of the ideas
developed in Kusuoka \cite{kusuoka}, summarised in Figure 1. A stochastic
Taylor expansion of $f(\xi_{x,s})$ can be written as $\Gamma\left(  \pi
_{m}(S_{0,1}(\circ B))\right)  f\left(  x\right)  ,$ i.e. the differential
operator obtained from the truncated signature under the map $\Gamma$ acting
on $f$ at $x$. As the signature is taking values in the tensor algebra we may
call the Taylor approximation the tensor level (approximation). It follows
immediately from the definition of a degree $m$ cubature measure on Wiener
space that the expectation of the degree $m$ Taylor approximation under $P$
and $Q$ is identical. Although the actual cubature step (exchanging the
measures $P$ and $Q$) has to take place at the tensor level we cannot do it
directly as the error bounds would involve higher derivatives in the direction
of $V_{0}$, which we have set out to avoid. Instead we follow Kusuoka
\cite{kusuoka} and observe that the signature may be written as the
exponential of the log signature and by interchanging $\exp$ and $\Gamma$ we
obtain a new approximation at the level of the flow by $f\left(
\text{Exp}\left[  \Gamma\left(  \pi_{m}\log S_{0,1}(\circ B)\right)  )\right]
\left(  x\right)  \right)  .$ Lemma \ref{k3} formalises this statement and
allows us move between the tensor algebra and the flow level. Crucially at the
level of flows it suffices to approximate $\xi_{t,x}$ by $\hat{\xi}_{t,x}$ in
$L^{1}$ norm as the bound for the approximation of $f(\hat{\xi}_{t,x})$ is
only increased by a factor of $\Vert\nabla f\Vert_{\infty}$.\newline

A key observation Kusuoka exploits is that if the Lie polynomial defining the
flow does not involve a $\epsilon_{0}$ component the error bound for moving
between flow and tensor level does not involve higher $V_{0}$ derivatives. By
using a splitting argument at the level of flows (Lemma \ref{k2}) he can
replace the log-signature by a Baker-Campbell-Hausdorff style term that does
not involve $\epsilon_{0}$. This allows him to move to the tensor level and
complete the approximation without using higher $V_{0}$ derivatives. \newline

To apply Kusuoka's argument to cubature on Wiener space we will go through
this approximation process (the full lines in Figure 1) for both the Wiener
measure and the cubature measure. By using the defining cubature identity
(\ref{cubature}) we will then be able to see that the approximations at the
end of each chain agree and obtain the desired bound.\newline

\begin{figure}[ptb]
\begin{picture}(385,215)
\put(10,5){\framebox(175,30){ $E  f \Big( Exp \left[ \, \, \Gamma \, \pi_m \,  log S_{0,1}(\circ B)   \, \right](x) \Big) $ }}
\put(10,185){\framebox(150,30){$ E (f(\xi_{1,x})) $ }}
\put(250,185){\makebox(50,30)[l]{SDE}}
\put(230,125){\makebox(50,30)[l]{Tensor level}}
\put(230,5){\makebox(50,30)[l]{Flow level}}
\put(40,125){\framebox(150,30){$ E \left( \Gamma \left( \pi_m \, S_{0,1}(\circ B) \right) f (x)  \right) $ }}
\put(75,35){\framebox(290,33){$ E  \bigg[ f\left(Exp\left[\Gamma\left(\pi_m log( exp (-\epsilon_0) \, S_{0,1} (\circ B) ) \right)\right](Exp(V_0)(x))\right)
\bigg] $ }}
\put(80,92){\framebox(275,33){$ E  \bigg( \left[ \Gamma\ \pi_m (exp (-\epsilon_0) \, S_{0,1} (\circ B))  \right] f(Exp(V_0)(x)) \bigg) $ }}
\put(30,185){\vector(0,-1){150}}
\put(55,45){\vector(0,-1){10}}
\put(55,55){\line(0,1){10}}
\put(55,75){\line(0,1){10}}
\put(55,95){\line(0,1){10}}
\put(55,115){\vector(0,1){10}}
\put(195,68){\vector(0,1){24}}
\put(75,68){\makebox(50,24)[l]{Lemma \ref{k3}}}
\put(-15,105){\makebox(50,24)[l]{Lemma \ref{k1}}}
\put(90,180){\line(0,-1){8}}
\put(90,165){\vector(0,-1){10}}
\end{picture}
\end{figure}The following two lemmas may be found in Kusuoka \cite{kusuoka}
(Corollary 15 and 17) and we will state them without proof.

\begin{lemma}
\label{k1} Let $m \geq1$, then there exists $C>0$ such that
\[
\Big| E_{P}\left(  f(\xi_{s,x})\right)  - E_{P} \left\{  f \Big( Exp \left[
\, \, \Gamma\, \pi_{m} \, \langle\sqrt{s}, log S_{0,1}(\circ B) \rangle\,
\right]  (x)\Big) \right\}  \Big| \leq C s^{(m+1)/2} \| \nabla f \|_{\infty}%
\]
for any $x \in R^{N}$, $s \in(0,1]$ and $f \in C_{b}^{\infty}(R^{N})$.
\end{lemma}

The second lemma is the splitting argument at the level of flows mentioned in
the previous discussion.

\begin{lemma}
\label{k2}

Let $m \geq2$ and $L^{(i)}$, $i=1,2$ denote two $\mathcal{L}^{(m)}$ valued
random variables with $E[ \| \pi_{k}(L^{(i)}) \|_{2}^{k} ] < \infty$ for any
$k \geq1$. Then for any $m \geq1$ and $p \in[1,\infty)$, there is $C>0$ such
that
\[
\bigg\| Exp(\Gamma\pi_{m} \langle\sqrt{s},L^{(1)} \rangle) \left(  Exp\left[
\Gamma\pi_{m} \langle\sqrt{s},L^{(2)} \rangle\right]  \,(x)\,\right)
\]%
\[
- Exp\left[  \Gamma\left(  \pi_{m} log( exp \langle\sqrt{s},L^{(2)}\rangle\,
exp \langle\sqrt{s},L^{(1)}\rangle) \right)  \right]  (x) \bigg\|_{L^{p}} \leq
C s^{(m+1)/2}%
\]
for all $s \in(0,1]$ and $x \in R^{N}$.
\end{lemma}

Note that
\[
Exp(\Gamma\pi_{m} \langle\sqrt{s},L^{(1)} \rangle) \left(  Exp\left[
\Gamma\pi_{m} \langle\sqrt{s},L^{(2)} \rangle\right]  \,(x)\,\right)
\]
is the composition of $Exp(\cdot)$ functions.\newline

The following Lemma bounds the difference between flow and tensor
approximation. It improves on \cite{kusuoka} by considering a different
truncation of the Taylor approximation.

\begin{lemma}
\label{k3} Let $w=\sum_{i=1}^{m}w_{i}\in$ $\pi_{m}(\mathcal{L})$, such that
each $w_{i}\in\left(  \pi_{i}-\pi_{i-1}\right)  (\mathcal{L}),$ (i.e. each
$w_{i}$ is a homogeneous Lie polynomial of degree $i)$ and $m\geq1$ then
\[
\sup_{x\in R^{N}}\left\vert f\big(Exp[\Gamma(w)](x)\big)-\left(  \Gamma\left[
\pi_{m}exp\left(  w\right)  \right]  f\right)  (x)\right\vert \leq\sum
_{j=1}^{m}\left\Vert \Gamma\left(  (\pi_{2m}-\pi_{m})w^{\otimes j}\right)
f\right\Vert _{\infty}%
\]
If moreover $w$ satisfies $(w,\epsilon_{0})=0$ then there is a constant $C>0$
such that
\[
\sup_{x\in R^{N}}\left\vert f(Exp[\Gamma\langle\sqrt{s},w\rangle
](x)-\Gamma\big(\pi_{m}exp(\langle\sqrt{s},w\rangle)f\big)(x)\right\vert \leq
C\sum_{j=m+1}^{2m}s^{j/2}\Vert f\Vert_{V,j}%
\]
for any $s\in(0,1]$, $x\in R^{N}$.
\end{lemma}

\begin{proof}
We first proceed as in Proposition 9 of \cite{kusuoka} by noting that
$\Gamma(w)$ is a smooth vector field defined on all of $R^{N}$. Thus for any
smooth $f$ we have
\[
\frac{d}{dt}f(Exp(t\Gamma(w))(x))=((\Gamma w)f)(Exp(t\Gamma(w))(x)).
\]
Using this identity iteratively to expand $f(exp(t\Gamma(w))(x))$ in a Taylor
expansion one sees that
\begin{align*}
&  f(Exp(t\Gamma(w))(x))-\sum_{j=0}^{m}\frac{t^{j}}{j!}\left(  \Gamma\left(
\sum_{i_{1}+\cdots+i_{j}=0}^{m}w_{i_{1}}\otimes\cdots\otimes w_{i_{j}}\right)
\right)  f(x)\\
&  =\sum_{i_{2}+\cdots+i_{j}\leq m,\text{ }i_{1}+\cdots+i_{j}>m}\int_{0}%
^{t}\frac{\left(  t-s\right)  ^{j}}{j!}\left[  \Gamma\left(  w_{i_{1}}%
\otimes\cdots\otimes w_{i_{j}}\right)  f\right]  \left(  \exp\left[
s\Gamma\left(  w\right)  \right]  \left(  x\right)  \right)  ds.
\end{align*}
Setting $t$ to one we deduce that
\begin{align*}
&  \left\vert f(Exp(\Gamma(w))(x))-\left(  \Gamma\left(  \sum_{i_{1}%
+\cdots+i_{j}=0}^{m}\frac{1}{j!}w_{i_{1}}\otimes\cdots\otimes w_{i_{j}%
}\right)  \right)  f(x)\right\vert \\
&  \leq\left\Vert \sum_{i_{2}+\cdots+i_{j}\leq m,\text{ }i_{1}+\cdots+i_{j}%
>m}\frac{1}{j!}\Gamma\left(  w_{i_{1}}\otimes\cdots\otimes w_{i_{j}}\right)
f\right\Vert _{\infty}\\
&  \leq\sum_{j=1}^{m}\left\Vert \Gamma\left(  (\pi_{2m}-\pi_{m})w^{\otimes
j}\right)  f\right\Vert _{\infty}%
\end{align*}
Noting that $\pi_{m}\exp\left(  w\right)  =\sum_{i_{1}+\cdots i_{j}=0}%
^{m}\frac{1}{j!}w_{i_{1}}\otimes\cdots\otimes w_{i_{j}}$ yields the first
claim. The second follows by considering $\langle\sqrt{s},w\rangle$ in place
of $w$ and noting that under the assumption $(w,\epsilon_{0})=0$, the vector
field $V_{0}$ does not appear on its own in the composition of the
differential operators on the right hand side of the last inequality.
\end{proof}

The last lemma is obtained by combining arguments from Lyons, Victoir
\cite{lyons} and Kusuoka \cite{kusuoka}.

\begin{lemma}
\label{l1} Let $t\in(0,1]$ and $Q_{t}$ be a cubature measure for Wiener space.
Then for any $x\in R^{N}$
\[
\Big|E_{Q_{t}}\left(  f(\xi_{t,x})\right)  -E_{Q_{1}}\left\{  f\Big(Exp\left[
\,\,\Gamma\,\pi_{m}\,\langle\sqrt{t},logS_{0,1}(\circ B)\rangle\,\right]
(x)\Big)\right\}  \Big|\leq Ct^{(m+1)/2}\Vert\nabla f\Vert_{\infty}%
\]
for all $f\in C_{b}^{\infty}(R^{N})$, where $C$ is a constant independent of
$t$ and $f$.
\end{lemma}

\begin{proof}
Let $g\in C_{b}^{\infty}(R^{N})$. By Lemma \ref{lem1.1} we may write
\begin{equation}
E_{Q_{t}}\left\vert g(\xi_{t,x})-\Gamma\big(\pi_{m}exp(\log S_{0,t}(\circ
B))\big)g(x)\right\vert \leq C\sum_{j=m+1}^{m+2}t^{j/2}\sup_{\alpha\in
A(j)\setminus A(j-1)}\Vert V_{\alpha}g\Vert_{\infty}. \label{lem:eq1a}%
\end{equation}
Letting $w=\pi_{m}\log S_{0,t}(\circ B)$ and applying Lemma \ref{k3} we see
that
\begin{equation}
E_{Q_{t}}\left\vert g(Exp[\Gamma w](x)-\Gamma\big(\pi_{m}%
exp(w)\big)g(x)\right\vert \leq C\sum_{j=m}^{2m}t^{j/2}\sup_{\alpha\in
A(j)\setminus A(j-1)}\Vert V_{\alpha}g\Vert_{\infty} \label{lem:eq2}%
\end{equation}
Combining (\ref{lem:eq1a}) and (\ref{lem:eq2}) with $g$ the identity function
we see that
\[
E_{Q_{t}}\Big|\xi_{t,x}-Exp\left[  \,\,\Gamma\,\pi_{m}\,logS_{0,t}(\circ
B)\,\right]  (x)\Big|\leq C\sum_{j=m}^{2m}t^{j/2}\sup_{\alpha\in A(j)\setminus
A(j-1)}\Vert V_{\alpha}Id\Vert_{\infty}%
\]
and the lemma follows.
\end{proof}

We are now ready to prove Theorem \ref{mainthm}. Our proof is modelled along
Kusuoka \cite{kusuoka} Lemma 18. We will go through a sequence of
approximations for the expectation of $f(\xi_{s,x})$ under the Wiener measure
and the cubature measure. Finally we show that the approximations at each end agree.

\begin{proof}
{(Theorem \ref{mainthm})}

Let $\mu_{1,s}=P$ be the Wiener measure on paths parametrised over $[0,s]$ and
$\mu_{2,s} = Q_{s}$. From Lemmas \ref{k1} and \ref{l1} we see that for
$i=1,2$
\begin{equation}
\label{eq2}\sup_{x \in R^{N}} \Big| E_{\mu_{i,s}}\left(  f(\xi_{s,x})\right)
- E_{\mu_{i,1}} \left\{  f \Big( Exp \left[  \, \, \Gamma\, \pi_{m} \,
\langle\sqrt{s}, \log S_{0,1}(\circ B) \rangle\, \right]  (x) \Big) \right\}
\Big| \leq C s^{(m+1)/2} \| \nabla f \|_{\infty}%
\end{equation}

Let
\[
L^{(1)}=\pi_{m}\log S_{0,1}^{(m)}(\circ B)\quad and\quad L^{(2)}=-\epsilon
_{0}.
\]
It is well known (see e.g. Lyons \cite{lyons2}) that the log-signature of
Brownian motion is a Lie series with probability one. Also we have $E[\Vert
\pi_{k}(logS_{0,1}(\circ B))\Vert_{2}^{k}]<\infty$, in fact using the
techniques of rough paths a similar statement can be obtained at the level of
paths. For example Lyons and Sidorova compute in \cite{sidorova} the radius of
convergence for the log signature. \newline

Hence, Lemma \ref{k2} implies that for $i =1,2$
\begin{equation}
\label{eq1a}\bigg| E_{\mu_{i,1}} f \Big( Exp(\Gamma\pi_{m} \langle\sqrt
{s},L^{(1)} \rangle) \left(  Exp\left[  \Gamma\pi_{m} \langle\sqrt{s},L^{(2)}
\rangle\right]  \,(z)\,\right)  \Big)
\end{equation}
\[
- E_{\mu_{i,1}} f \Big( Exp\left[  \Gamma\pi_{m} log\left(  exp \langle
\sqrt{s},L^{(2)}\rangle\, exp \langle\sqrt{s},L^{(1)}\rangle\right)  \right]
(z) \Big) \bigg| \leq C s^{(m+1)/2} \| \nabla f \|_{\infty}.
\]
Writing
\[
x=Exp\left(  \Gamma\langle\sqrt{s}, -\epsilon_{0}\rangle\right)  (z)
\]
the inequality (\ref{eq1a}) becomes
\begin{multline}
\label{eq1}\bigg| E_{\mu_{i,1}} f \Big( Exp(\Gamma\pi_{m} \langle\sqrt
{s},L^{(1)} \rangle) (x) \Big)\\
- E_{\mu_{i,1}} f \Big( Exp\left[  \Gamma\pi_{m} log\left(  exp \langle
\sqrt{s},L^{(2)}\rangle\, exp \langle\sqrt{s},L^{(1)}\rangle\right)  \right]
(Exp\left(  \Gamma\langle\sqrt{s}, \epsilon_{0}\rangle\right)  (x))
\Big) \bigg|\\
\leq C s^{(m+1)/2} \| \nabla f \|_{\infty}.
\end{multline}

Thus, combining inequalities (\ref{eq2}) and (\ref{eq1}) and using the
triangle inequality we see that
\begin{multline}
\sup_{x\in R^{N}}\Big|E_{\mu_{i,s}}f(\xi_{s,x})\label{eq3}\\
-E_{\mu_{i,1}}f\Big(Exp\left[  \Gamma\pi_{m}log\Big(exp(\langle\sqrt
{s},-\epsilon_{0}\rangle)\,\langle\sqrt{s},S_{0,1}(\circ B)\rangle
\Big)\right]  (Exp\left(  \Gamma\langle\sqrt{s},\epsilon_{0}\rangle\right)
(x))\Big)\bigg|\\
\leq Cs^{(m+1)/2}\Vert\nabla f\Vert_{\infty}.
\end{multline}

It follows from the Baker-Campbell-Hausdorff formula that
\begin{equation}
\pi_{m}log\Big(exp(\langle\sqrt{s},-\epsilon_{0}\rangle)\,\langle\sqrt
{s},S_{0,1}(\circ B)\rangle\Big) \label{eq3a}%
\end{equation}
has no $\epsilon_{0}$ component, i.e.
\[
\left(  \pi_{m}log\Big(exp(\langle\sqrt{s},-\epsilon_{0}\rangle)\,\langle
\sqrt{s},S_{0,1}(\circ B)\rangle\Big),\epsilon_{0}\right)  =0,
\]
where we defined the inner product in (\ref{inner-product}).\newline Moreover
as the log signature of the Brownian motion is a Lie series with probability
one, $\left(  \ref{eq3a}\right)  $ is a Lie polynomial.

Hence, we may apply Lemma \ref{k3} to inequality (\ref{eq3}) and once again
using the triangle inequality we obtain for $i=1,2$
\begin{equation}
\sup_{x\in R^{N}}\Big|E_{\mu_{i,s}}\left(  f(\xi_{s,x})\right)  -E_{\mu_{i,1}%
}\left(  \big(\Gamma\,\pi_{m}\,exp\left\langle \,\sqrt{s},\pi_{m}%
log\big(exp(-\epsilon_{0})\,S_{0,1}(\circ B)\big)\,\right\rangle
f\big)(y)\right)  \Big| \label{eq4}%
\end{equation}%
\[
\leq C\left(  \sum_{j=m+1}^{2m}s^{j/2}\Vert f\Vert_{V,j}+s^{(m+1)/2}%
\Vert\nabla f\Vert_{\infty}\right)  ,
\]
where $y=Exp\left(  \Gamma\langle\sqrt{s},\epsilon_{0}\rangle\right)  (x)$.
Note that in the previous step we have used the fact that the scaling
operation $\langle s,\cdot\rangle$ commutes with log and exp. We have also
used the fact that
\[
\pi_{m}\,exp\left\langle \,\sqrt{s},\pi_{m}log\big(exp(-\epsilon_{0}%
)\,S_{0,1}(\circ B))\right\rangle =\pi_{m}\langle\sqrt{s},exp(-\epsilon
_{0})S_{0,1}(\circ B)\rangle.
\]

Finally using the cubature relation (\ref{cubature})
\[
E_{P}\left(  \pi_{m}\,S_{0,1}(\circ B)\right)  =E_{Q_{1}}(\pi_{m}%
\,S_{0,1}(\circ B)))
\]
and noting that the multiplication by a deterministic tensor can be taken out
of the expectation we have
\[
E_{P}\Big(\pi_{m}(exp(-\epsilon_{0})S_{0,1}(\circ B))\Big)=E_{Q_{1}}%
\Big(\pi_{m}(exp(-\epsilon_{0})S_{0,1}(\circ B))\Big).
\]
Hence, it follows that
\[
E_{P}\left[  \big(\Gamma\,\pi_{m}\langle\sqrt{s},exp(-\epsilon_{0}%
)S_{0,1}(\circ B)\rangle f\big)(y)\right]  =E_{Q_{1}}\left[  \big(\Gamma
\,\pi_{m}\langle\sqrt{s},exp(-\epsilon_{0})S_{0,1}(\circ B)\rangle
f\big)(y)\right]  .
\]
Using this identity in $\left(  \ref{eq4}\right)  $ a final application of the
triangle inequality completes the proof of the theorem.
\end{proof}

\begin{remark}
\label{klv-kusuoka} The truncated log signatures of the cubature paths of a
degree $m$ cubature measure satisfy the definition of a $m$-$\mathcal{L}%
$-moment similar random variable of Kusuoka \cite{kusuoka} with respect to the
truncated log signature of the Brownian motion. Conversely for any finite such
family we can find paths that satisfy a degree $m$ cubature formula. The
approximation operator for $P_{s}f$ whose error bounds are analysed in
\cite{kusuoka} can be written as
\[
E_{Q_{\mathcal{L}}}f\left(  Exp[\Gamma\langle\sqrt{s},\pi_{m}L\rangle
](x)\right)
\]
and it is clear from our discussion that the same bounds as in Theorem
\ref{mainthm} can be obtained for this approximation.
\end{remark}

\end{document}